\documentclass[12pt]{article}
\usepackage{graphicx}
\usepackage{multicol,multirow}
\usepackage{amsmath,amssymb,amsfonts}
\usepackage{mathrsfs}
\usepackage{amsthm}
\usepackage{rotating}
\usepackage{appendix}
\usepackage[numbers]{natbib}
\usepackage{ifpdf}
\usepackage{xcolor}
\usepackage[colorlinks,allcolors=blue]{hyperref}
\usepackage{hyperref}
\usepackage{color}
\usepackage{amssymb,amsmath}
\usepackage{graphicx}
\usepackage{orcidlink}
\newtheorem{theorem}{Theorem}
\newtheorem{lemma}{Lemma}
\newtheorem{corollary}{Corollary}

\newtheorem{proposition}{Proposition}

\newtheorem{remark}{Remark}

\title{\bf On  generalized Stirling numbers and special functions  }
\author{ Kamel  Mezlini $^{a\orcidlink{0000-0001-6983-1150}}$ \& Tahar Moumni $^*$ $^{b\orcidlink{0009-0009-3853-9603}}$ \& Najib Ouled Azaiez $^{c\orcidlink{0000-0002-0082-4611}}$
}

\begin{document}
 \noindent \maketitle
        \begin{center}
            $^a$ Department of Mathematics, Faculty of Sciences of Tunisia,  University of Elmanar, 2092 Tunis El-Manar, Tunisia.\\
            $^b$ Department of Mathematics, Faculty of Sciences of Bizerte,  University of Carthage, 7021 Jarzouna,  Tunisia.\\
            $^{c}$ Department of Mathematics, Faculty of Sciences of Sfax,  University of Sfax, B.P 1171, Sfax 3000 Tunisia\\
            $^*$ Corresponding Author: tahar.moumni@issatm.ucar.tn
        \end{center}

\vskip 0.5cm

\begin{abstract}
 We introduce a new generalization of Stirling numbers of the second kind and analyze their properties, including generating functions, integral representations, and recurrence relations. These numbers are used to approximate Riemann zeta values by rationals with exponentially decreasing error. We establish connections with Hurwitz zeta functions, polylogarithms, harmonic sums, and multiple sums. Finally, we extend our study to q-Stirling numbers, linking them to q-hypergeometric functions and a q-zeta function, revealing new insights in combinatorics and number theory.
\end{abstract}


\section{Introduction}

Stirling numbers of the first kind $s(n, k)$ and of the second kind $S(n, k)$ frequently appear in combinatorics and analysis. They are defined by the following generating functions (see [\cite{Comtet}, Chapter 5]):

\begin{equation}
    \frac{\ln^k(1 + t)}{k!} = \sum_{n=k}^{\infty} s(n, k) \frac{t^n}{n!}, \quad |t| < 1, \quad k = 0, 1, 2, \dots
\end{equation}

\begin{equation}
    \frac{(e^t - 1)^k}{k!} = \sum_{n=k}^{\infty} S(n, k) \frac{t^n}{n!}, \quad t \in \mathbb{R}, \quad k = 0, 1, 2, \dots
\end{equation}

They have been extensively studied by several renowned mathematicians, including Euler (1755), Laplace (1812), and Cayley (1888). Stirling numbers were  generalized by numerous mathematicians due to their wide range of applications in combinatorics, number theory, and mathematical analysis. For example, they played a crucial role in providing a representation of the Riemann zeta function, as demonstrated in \cite{Butzer91}. In 2007, Everitt et al. introduced the so-called Jacobi-Stirling numbers of the first and second kind (see \cite{Evritt07}), further enriching their theoretical significance. More recently, Nenad P. et al. have extended the study of Jacobi-Stirling numbers, broadening their applicability (see \cite{Gomaa}).

In this paper, we introduce a new generalization of Stirling numbers of the second kind and analyze their fundamental properties, including their horizontal and vertical generating functions, integral representations, and recurrence relations. We then establish a connection between these numbers and certain special functions, such as the \textbf{Hurwitz zeta function} and \textbf{polylogarithm functions}. Special attention is given to the approximation of Riemann zeta function values using sequences of rational numbers, where we show that the error in this approximation decreases exponentially (see Theorem 7).

Furthermore, we investigate the links between these generalized numbers and various combinatorial sums, such as \textbf{harmonic sums} and \textbf{multiple sums}. Finally, we explore a natural extension of the generalized Stirling numbers within the framework of \textbf{q-deformation}, thereby establishing new connections with q-deformed hypergeometric functions and q-generalized zeta functions.

The structure of this paper is as follows. \textbf{Section 2} introduces the definition of generalized Stirling numbers of the second kind and their analytical properties. \textbf{Section 3} establishes a relationship between these numbers and the Riemann zeta function, proposing an efficient approximation of its values at integers. \textbf{Sections 4 and 5} focus on the series expansions of the \textbf{Hurwitz zeta function} and \textbf{polylogarithm functions}. \textbf{Sections 6 and 7} respectively explore the connections with \textbf{harmonic sums} and \textbf{multiple sums}. Finally, \textbf{Section 8} is dedicated to the q-deformed generalization of Stirling numbers, illustrating their connection with hypergeometric series and q-generalized zeta functions.

\section{Generalized Stirling numbers of the second type }\label{GSN}
The generalized Stirling numbers of the second kind are defined by the following:  
\begin{equation}\label{stir-expl-1}
S_n^p=\frac{(-1)^n}{n!}\sum_{k=0}^{n}\frac{(-1)^{k}\binom{n}{k}}{(k+1)^p},\;\;n=0,1,2,...,p=1,2,....
\end{equation}
\begin{theorem}\label{lemma2} The horizontal generating function associated to the generalized Stirling numbers of the second type $ S_n^p$  is given by:
\begin{equation}\label{part1-1}
e^{-t}{}{ }_{p} F_{p}\left[1,...,1;2,...,2 ; t\right]= \sum_{n=0}^{+\infty}
S_n^pt^{n},\;t\in\mathbb{R}.
\end{equation}
\end{theorem}
\begin{proof}
Recall from \cite{Bailey}, that  the generalized hypergeometric
function is defined by:
 
\begin{equation}\label{fpq}
{ }_{p} F_{q}\left[a_{1}, a_{2}, \ldots, a_{p}; b_{1}, b_{2},
\ldots, b_{q} ; z \right]=\sum_{k=0}^{\infty}
\frac{\prod_{m=1}^{p}\left(a_{m}\right)_{k}
z^{k}}{\prod_{m=1}^{q}\left(b_{m}\right)_{k} k !},
\end{equation}
where $(c)_{n}$ is the shifted factorial  defined by
$$
(c)_{0}=1 \quad \text { and } \quad(c)_{n}=c(c+1) \cdots(c+n-1)
\quad \text { for } n=1,2, \ldots
$$
Using  power series expansion of both functions on the left hand side of
(\ref{part1-1}), we get:
\begin{eqnarray*}
 e^{-t}{ }_{p} F_{p}\left[1,...,1;2,...,2 ; t\right]
  &=& \left( \sum_{n=0}^{+\infty}\frac{(-1)^nt^n}{n!}\right) \left( \sum_{k=0}^{+\infty}\frac{t^k}{k!(k+1)^p}\right)  \\
  &=&  \sum_{n=0}^{+\infty}\frac{(-1)^n}{n!}\sum_{k=0}^{+\infty}\frac{t^{n+k}}{k!(k+1)^p} \\
&=&  \sum_{n=0}^{+\infty} ( \sum_{k=0}^{n}\frac{(-1)^{n-k}}{(n-k)!k!(k+1)^p} ) t^n \\
 &=& \sum_{n=0}^{+\infty}S_n^pt^n.
\end{eqnarray*}
\end{proof}
Using the fact that $\displaystyle{\binom{n}{k-1}=\frac{k}{n+1} \binom{n+1}{k},}$
the generalized Stirling numbers can be written as
\begin{equation*}
  S_n^p=\frac{(-1)^{n+1}}{(n+1)!}\sum_{k=1}^{n+1}\frac{(-1)^{k}\binom{n+1}{k}}{k^{p-1}}=S(1-p,n+1),
\end{equation*}
where $S(\alpha,n)$ is the Stirling number of negative order defined by  Butzer, Kilbas and Trujillo, see \cite{Butzer}.\\
Using (\ref{fpq}), the  $S_n^p $ numbers can be  expressed as special values of a  hypergeometric function:
\begin{equation*}
  S_n^p=\frac{(-1)^{n}}{n!}{ }_{p+1} F_{p}\left[-n,1,...,1;2,...,2 ; 1\right].
\end{equation*}
\begin{theorem}
The generalized Stirling numbers of the second kind
 have the integral representation:
 \begin{equation}\label{integral-rep}
 S_n^p=\frac{(-1)^n}{n!(p-1)!}\int_{0}^{+\infty}(1-e^{-t})^n e^{-t}t^{p-1}dt,\;\;\;n=0,1,2,...,\;p=1,2,3,....
 \end{equation}
 \end{theorem}
\begin{proof}
By substituting Euler's integral:
\begin{equation}\label{euler-int}
 \int_{0}^{+\infty}e^{-t(k+1)}t^{p-1}dt=\frac{(p-1)!}{(k+1)^p},
\end{equation}
 in (\ref{stir-expl-1}), the result (\ref{integral-rep}) follows.\end{proof}
\begin{remark}
From the integral representation (\ref{integral-rep}), we have the following properties:
\begin{equation}\label{snp-estm1-1}
\forall\;n\in\mathbb{N},\;\;|S_n^p|\le \frac{1
}{n!}\;\;\mbox{and}\;\; \lim_{n\rightarrow+\infty}n!S_n^p=0.
\end{equation}
\end{remark}

\begin{remark}
By change of variable $u=e^{-t}$ in (\ref{integral-rep}), we deduce
\begin{equation}\label{stir-ln2}
S_n^p=\frac{(-1)^{n+p-1}}{n!(p-1)!}\int_{0}^{1}(1-u)^n\ln(u)^{p-1}du=\frac{(-1)^{n+p-1}}{n!(p-1)!}\int_{0}^{1}u^n\ln(1-u)^{p-1}du.
\end{equation}
\end{remark}
Integration by parts of the integral in (\ref{integral-rep})
produces the recurrence for $S_{n}^{p}.$
\begin{theorem}
The generalized Stirling numbers of the second kind
satisfy the recurrence relation
\begin{equation*}
  (n+1)S_{n}^{p+1}=S_{n}^{p}-S_{n-1}^{p+1},\;n,\;p=1,2,3,...,
\end{equation*}
with
\begin{equation*}
S_{0}^{p}=1;\;S_{n}^{1}=\frac{(-1)^{n+1}}{(n+1)!}.
\end{equation*}
\end{theorem}
\begin{proof} Integration by parts of the integral in
(\ref{integral-rep}) produces
\begin{equation}\label{snp-int-part}
S_{n}^{p}=  -\frac{(-1)^{n}}{(n-1)!p!}\int_{0}^{+\infty}(1-e^{-t})^{n-1}e^{-2t} t^{p}dt+\frac{(-1)^{n}}{n!p!}\int_{0}^{+\infty}(1-e^{-t})^{n}e^{-t} t^{p}dt.
\end{equation}
So,
\begin{equation}\label{sn-1p}
S_{n}^{p}=   -\frac{(-1)^{n}}{(n-1)!p!}\int_{0}^{+\infty}(1-e^{-t})^{n-1}(e^{-t}-1+1)e^{-t} t^{p}dt+S_{n}^{p+1}
\end{equation}
then $S_{n}^{p} $ can be expressed in the form
\begin{equation}\label{snpdiv1}
S_{n}^{p}=   \frac{(-1)^{n}}{(n-1)!p!}\int_{0}^{+\infty}(1-e^{-t})^{n}e^{-t} t^{p}dt-\frac{(-1)^{n}}{(n-1)!p!}\int_{0}^{+\infty}(1-e^{-t})^{n-1}e^{-t} t^{p}dt+S_{n}^{p+1}
\end{equation}
From (\ref{snpdiv1}) and (\ref{integral-rep}), we deduce that:
\begin{equation*}
  S_{n}^{p}=(n+1)S_{n}^{p+1}+S_{n-1}^{p+1}.
\end{equation*}
\end{proof}

Let us  consider the function
\begin{equation}\label{gnpoch}
 g_{n}(x)=\frac{(-1)^{n}}{(x)_{n+1}},\;x>0,\;\;\;n=0,1,2....
\end{equation}
The following proposition give us other expressions of $g_{n}$ and
its relation with the generalized sitirling numbers of the second
kind.
\begin{proposition}
\begin{enumerate}
\item
The integral representation of the  function $g_n$ is given by:
\begin{equation}\label{g-int-1}
g_{n}(x)=\frac{(-1)^{n}}{n!}\int_{0}^{+\infty}(1-e^{-t})^n e^{-xt}dt,\;\;x>0.
\end{equation}
\item The function $g_n$ can be expressed as
\begin{equation}\label{gnsnx}
 g_{n}(x)= \frac{(-1)^{n}}{n!}\sum_{k=0}^{n}\frac{(-1)^{k}\binom{n}{k}}{k+x}.
\end{equation}
\item The number $S_n^p$ can be expressed  in terms of values  of  derivatives of the function $g_n:$
\begin{equation}\label{deri-gn}
  S_n^p=\frac{(-1)^{p-1}}{(p-1)!} g_{n}^{(p-1)}(1),
\end{equation}
where  $g_n^{(p-1)}$ denotes the $(p-1)$-th derivative of $g_n$.
\end{enumerate}
\end{proposition}
\begin{proof}
\begin{enumerate}
\item Using \cite{andrews},   the change of variables $u=e^{-t} $ in  (\ref{g-int-1}),
leads to:
\begin{equation*}
\frac{(-1)^{n}}{n!}\int_{0}^{1}(1-u)^n u^{x-1}du=\frac{(-1)^n}{n!}\beta(x,n+1)=\frac{(-1)^{n}}{n!}\frac{\Gamma(x)\Gamma(n+1)}{\Gamma(x+n+1)}.
\end{equation*}
Since $ \Gamma(n+1)=n! $ and  $
(x)_{n+1}=\frac{\Gamma(x+n+1)}{\Gamma(x)}$, (\ref{g-int-1}) follows.
\item Using the binomial expansion  of $(1-e^{-t})^n$ in  (\ref{g-int-1})  and  the fact that
$ \displaystyle{\int_{0}^{+\infty}e^{-t(k+x)}dt=\frac{1}{k+x},}$
we deduce (\ref{gnsnx}), see \cite{Gould}, p.82.
\item
The proof of (\ref{deri-gn}) is based on  the $(p-1)-$time differentiation  of
(\ref{g-int-1}) with respect to $x.$ Then taking $x=1$ leads to:
\begin{equation*}
  g_{n}^{(p-1)}(1)=\frac{(-1)^{n+p-1}}{n!}\int_{0}^{+\infty}(1-e^{-t})^n e^{-t}t^{p-1}dt=(-1)^{p-1}(p-1)!S_n^p.
\end{equation*}
\end{enumerate}
\end{proof}

The following theorem borrowed from \cite{preprint-Kamel}, gives another generating function of the generalized Stirling numbers  of the second kind.
\begin{theorem}
The generalized Stirling numbers  of the second kind
can be defined by the following vertical generating function:
\begin{equation}\label{generatp-1}
  \frac{(-1)^{n}}{(1-t)_{n+1}}=\sum_{p=0}^{+\infty}S_n^{p+1}t^p,\; 0<t<1.
\end{equation}
\end{theorem}
\begin{proof} To prove (\ref{generatp-1}), we write $g_n(x) $ as a power
series in $(x-1)$. We use Taylor's theorem to write
\begin{equation*}
g_n(x)=  \sum_{p=0}^{+\infty}\frac{g_{n}^{(p)}(1)}{p!}(x-1)^p.
\end{equation*}
Hence, from equation (\ref{deri-gn}), we have
\begin{equation}\label{gn-stir}
g_n(x)=  \sum_{p=0}^{+\infty}S_n^{p+1}(1-x)^p.
\end{equation}
Putting $t=1-x$ in (\ref{gn-stir}), we get (\ref{generatp-1}).
\end{proof}

\section{The \texorpdfstring{$S_n^{p} $}- sequence and Riemann  zeta function}\label{REL}
In this section, we present a relation between the generalized Stirling numbers of the second kind and Riemann zeta functions.

\subsection{The \texorpdfstring{$S_n^{p}$}- sequence  and Riemann  zeta values }
To establish the relation between the $S_n^{p}$-sequence  and  zeta
values we recall first the following result established in
 \cite{preprint-Kamel}
\begin{theorem}
For $p\geq 2$ and for all $t \geq 0,$ we have
\begin{equation}\label{delta}
{}_pF_{p}\left[1,...,1;2,...,2; t \right]=e^{t} \sum_{n=0}^{+\infty}
\frac{(-1)^ns(n+p-1, p-1)}{(n+p-1)!t^{n+p}}\gamma(n+p,t).
\end{equation}
where $\gamma$ denoting the incomplete Gamma function.
\end{theorem}
We note also that the Laplace transform of the generalized
hypergeometric function is given by, see \cite{Grad}
\begin{equation}\label{laplace-hyper1}
\begin{gathered}
\int_{0}^{\infty} e^{-s t}{ }_{p} F_{q}\left[a_1,...,a_{p} ;b_1,..., b_{q} ; t\right] dt=\frac{1}{s}{ }_{p+1} F_{q}\left[1, a_{1}, \ldots, a_{p} ; b_{1}, \ldots, b_{q} ; s^{-1}\right]. \\
{[p \leq q]}
\end{gathered}
\end{equation}
We recall also that the Riemann zeta function is related to the generalized
hypergeometric function as the following, (see \cite{andrews} page
106)

\begin{equation}\label{01}
\zeta(p)={}_{p+1}F_{p}\left[1,...,1;2,...,2; 1\right]
\end{equation}



\begin{theorem}
The following equality holds true for all integer $p\geq 2$ and for
all positive real parameter $R.$
\begin{equation}\label{zet-stir}
  \zeta(p)=\sum_{n=0}^{+\infty} \frac{S_n^pR^{n+1}}{n+1}+\sum_{n=0}^{\infty}\frac{(-1)^{n}s(n+p-1,p-1)}{(n+p-1)!}\left(\frac{\gamma(n+p,R)}{(n+p-1)R^{n+p-1}}+e^{-R}\right).
\end{equation}
\end{theorem}
\begin{proof} To prove (\ref{zet-stir}) we use 
(\ref{01}) together with (\ref{laplace-hyper1}), to obtain
\begin{eqnarray*}
 \zeta(p) &=&{}_{p+1}F_{p}\left[1,...,1;2,...,2; 1\right] \\
&=&\int_{0}^{\infty} e^{- t}{ }_{p} F_{p}\left[1,...,1 ; 2,...,2 ;
t\right] dt,
\end{eqnarray*}
or equivalently, using (\ref{part1-1}) and (\ref{delta}), we obtain
\begin{eqnarray*}
 \zeta(p) &=&\int_{0}^{R} e^{-t}{ }_{p} F_{p}\left[1,...,1 ; 2,...,2 ; t\right]dt+\int_{R}^{+\infty} e^{-t}{ }_{p} F_{p}\left[1,...,1 ; 2,...,2 ; t\right]
 dt\\
 &=&\underbrace{\int_{0}^{R} \sum_{n=0}^{+\infty}S_n^p t^{n}dt}_{I_1}+\underbrace{\int_{R}^{+\infty} \sum_{n=0}^{+\infty} (-1)^n\frac{s(n+p-1, p-1)}{(n+p-1)!t^{n+p}}\gamma(n+p,t)
 dt}_{I_2}.
\end{eqnarray*}

By  (\ref{snp-estm1-1}), we have $|S_n^p|\le\frac{1}{n !}$, from which
it follows that
\begin{eqnarray*}
\int_{0}^{R} \sum_{n=0}^{+\infty}|S_n^p| t^{n} d t & \leq& \int_{0}^{R} \sum_{n=0}^{+\infty} \frac{ t^{n}}{n !}dt =\int_{0}^{R} e^{ t} d t=e^{ R}-1.
\end{eqnarray*}
This fact justifies the interchange of integration and summation in $I_1$, and we have
\begin{equation*}
 I_1=\sum_{n=0}^{+\infty} \frac{S_n^p}{n+1} R^{n+1}.
\end{equation*}

By the monotone convergence theorem, we interchange summation and
integration in $I_2$ to obtain:
\begin{eqnarray*}
 I_2 &=&\sum_{n=0}^{\infty}\frac{(-1)^{n}s(n+p-1,p-1)}{(n+p-1)!} \underbrace{\int_{R}^{+\infty}\frac{\gamma(n+p,t)}{t^{n+p}}dt}_J.
\end{eqnarray*}
Integrating $J$ by part, we obtain
\begin{equation}\label{J}
\displaystyle{J=\frac{\gamma(n+p,R)}{(n+p-1)R^{n+p-1}}+\frac{e^{-R}}{n+p-1}}.
\end{equation}
Consequently
\begin{equation}\label{I22}
\displaystyle{I_2=\sum_{n=0}^{\infty}\frac{(-1)^{n}s(n+p-1,p-1)}{(n+p-1)(n+p-1)!}\left(\frac{\gamma(n+p,R)}{R^{n+p-1}}
+e^{-R}\right)}
\end{equation}
 which complete the proof of
(\ref{zet-stir}).
\end{proof}

\subsection{Approximating  Riemann zeta values by rationals}
Our aim now is to use (\ref{zet-stir}) to approximate zeta values by
sequence of rationals. To do so, we define the sequence $(\zeta_{N}(p))_N$ as follows:
For each integer $N\ge p $ ,
\begin{equation}\label{zeta(p)-app}
\zeta_{N}(p)=\sum_{n=0}^{4N}
\frac{S_n^p}{n+1}N^{n+1}+\sum_{n=0}^{N-p} \frac{(-1)^ns(n+p-1,
p-1)}{(n+p-1)N^{n+p-1}},
\end{equation}
 and we show that the error term  $|\zeta(p)-\zeta_{N}(p)|$  decay exponentially to zero for large $N$. Hence the value of $\zeta(p)$ is approximated by $\zeta_{N}(p)$ given by
 (\ref{zeta(p)-app}).
 \begin{theorem}\label{ZAPP}
 Their exist a positive constant $C$  such that for large  $N$,
we have the estimation
\begin{equation*}
 |\zeta(p)-\zeta_{N}(p)|\le CNe^{-N}
\end{equation*}
\end{theorem}
\begin{proof}
For larger integer $N$, and by taking $R=N$ in  (\ref{zet-stir}), we get
\begin{equation}\label{zet-stir2}
\begin{split}
 \zeta(p)-\zeta_{N}(p)=
\underbrace{\sum_{n=4N+1}^{+\infty}
\frac{S_n^pN^{n+1}}{n+1}}_{\epsilon_1}-\underbrace{\sum_{n=0}^{N-p}\frac{(-1)^{n}s(n+p-1,p-1)}{(n+p-1){N^{n+p-1}}}\left(1-\frac{\gamma(n+p,N)}{(n+p-1)!}\right)}_{\epsilon_2}&+\\
\underbrace{\sum_{n=N-p+1}^{\infty}\frac{(-1)^{n}s(n+p-1,p-1)}{(n+p-1)N^{n+p-1}}\frac{\gamma(n+p,N)}{(n+p-1)!}}_{\epsilon_3}
+\underbrace{e^{-N}\sum_{n=0}^{\infty}\frac{(-1)^{n}s(n+p-1,p-1)}{(n+p-1)(n+p-1)!}}_{\epsilon_4}.
\end{split}
\end{equation}
Using (\ref{snp-estm1-1}), we deduce that the first series in the right hand side of
(\ref{zet-stir2}) is estimated as follows
\begin{equation*}
\epsilon_{1} =
\displaystyle{\left|\sum_{n=4N+1}^{+\infty}\frac{S_n^p}{n+1}N^{n+1}\right|}
\leq \displaystyle{\sum_{n=4N+1}^{+\infty}\frac{N^{n+1}}{(n+1)!}},
\end{equation*}
and the error term is bounded by
\begin{equation*}
 \epsilon_1  \leq \displaystyle{\frac{N^{4N+2}}{(4N+2)!}\sum_{n=0}^{+\infty}\left(\frac{1}{4}\right)^n}=\displaystyle{\frac{4N^{4N+2}}{3(4N+2)!}}.
\end{equation*}
Using  the following well known bounds, valid for all positive integers $n$
$$
\sqrt{2 \pi n}\left(\frac{n}{e}\right)^{n} e^{\frac{1}{12 n+1}}<n
!<\sqrt{2 \pi n}\left(\frac{n}{e}\right)^{n} e^{\frac{1}{12 n}},
$$
we have
\begin{equation}\label{b}
\displaystyle{\frac{4N^{4N+2}}{3(4N+2)!}} \le \frac{4}{3\sqrt{2 \pi
(4N+2)}}\left(\frac{ eN}{4N+2}\right)^{4N+2} e^{-\frac{1}{12
(4N+2)+1}},
\end{equation}
\begin{equation*}
\epsilon_1 \le \frac{1}{\sqrt{ N}}\left(\frac{ e}{4}\right)^{4N}\le
\frac{1}{\sqrt{ N}}e^{-N}.
\end{equation*}
 Hence $\epsilon_1$ decay exponentially to $0$ when $N$ goes to infinity.
 To proceed further, we recall the following upper bound of the
Stirling number of the first kind  valid for $m\;
=\;1,\;\cdots\;,\;n -1,$ see \cite{Jose}
\begin{equation}\label{upper-stir1}
\displaystyle{|s(n+1, m+1)| \leq \frac{n !(\log n)^m}{m
!}\left(1+\frac{m}{\log n}\right)}.
\end{equation}
We recall also that, the following inequalities holds for all
integers $n,$ $p$ and $j$
\begin{equation}\label{Ineqfact1}
(n+p+j)!\geq (n+p)^j(n+p-1)!
\end{equation}
\begin{equation}\label{Ineqfact1}
(n+p+j)!\geq (n+p)!j!
\end{equation}
Using the following equality found in \cite{Temme}
\begin{equation}\label{equ-Temme}
\frac{\gamma(n+p,N)}{(n+p-1)!}=e^{-N}N^{n+p}\sum_{j=0}^{+\infty}\frac{N^j}{(n+p+j)!}.
\end{equation}
Since $n>N-p $ in series defining $\epsilon_3 $, and using the
following inequality
\begin{equation}\label{Ineqfact1}
(n+p+j)!\geq (n+p)^j(n+p-1)!,
\end{equation}
 we have
\begin{eqnarray*}
\frac{\gamma(n+p,N)}{(n+p-1)!}&\leq& e^{-N}N^{n+p}\sum_{j=0}^{+\infty}\frac{N^j}{(n+p)^j(n+p-1)!}\\
  &=& \frac{e^{-N}N^{n+p}}{(n+p-1)!}\sum_{j=0}^{+\infty}\left(\frac{N}{n+p}\right)^j\\
  &=& \frac{e^{-N}N^{n+p}(n+p)}{(n+p-1)!(n+p-N)}.
\end{eqnarray*}
Consequently,
\begin{eqnarray*}
  \epsilon_3
  &\leq&e^{-N}\sum_{n=N-p+1}^{\infty}\frac{(-1)^{n}s(n+p-1,p-1)N^{n+p}(n+p)}{(n+p-1)N^{n+p-1}(n+p-1)!(n+p-N)}\\
  &=& Ne^{-N}\sum_{n=N-p+1}^{\infty}\frac{(-1)^{n}s(n+p-1,p-1)(n+p)}{(n+p-1)(n+p-1)!(n+p-N)}.
\end{eqnarray*}
By (\ref{upper-stir1}),
\begin{eqnarray*}
\epsilon_{3}&\leq&\displaystyle{
\frac{Ne^{-N}}{(p-2)!}\sum_{n=N-p+1}^{\infty}\frac{(n+p)(\log(n+p-2))^{p-2}}{(n+p-1)^2(n+p-N)}\left(1+\frac{p-2}{\log(n+p-2)}\right)}\\
&\leq&\displaystyle{\frac{\beta Ne^{-N}}{(p-2)!}},
\end{eqnarray*}
where
$\displaystyle{\beta=\sum_{n=N-p+1}^{\infty}\frac{(n+p)(\log(n+p-2))^{p-2}}{(n+p-1)^2(n+p-N)}\left(1+\frac{p-2}{\log(n+p-2)}\right)}.$ This prove that $\epsilon_3$ decay exponentially
to $0.$
Using again (\ref{upper-stir1}) and by similar technics used to give
an upper bound  for $\epsilon_{3}$ we can show that
\begin{eqnarray*}
\epsilon_{4} & \leq& \displaystyle{
\frac{e^{-N}}{(p-2)!}\sum_{n=0}^{\infty}\frac{(\log(n+p-2))^{p-2}}{(n+p-1)^2}\left(1+\frac{p-2}{\log
(n+p-2)}\right)}\\
&\leq&\displaystyle{\frac{\mu e^{-N}}{(p-2)!}},
\end{eqnarray*}
where
$\displaystyle{\mu=\sum_{n=0}^{\infty}\frac{(\log(n+p-2))^{p-2}}{(n+p-1)^2}\left(1+\frac{p-2}{\log
(n+p-2)}\right)}.$ This prove that $\epsilon_4$ decay exponentially
to $0.$\\
Let now study the decay rate of $\epsilon_2$. Using
the following equality found in [\cite{Arfken}, Page 566]
\begin{equation}\label{arfken}
\displaystyle{1-\frac{\gamma(n,x)}{\Gamma(n)}=e^{-x}\sum_{n=0}^{n-1}\frac{x^j}{j!}},
\end{equation}
we can write
\begin{eqnarray*}
\epsilon_2 &=& e^{-N}\sum_{n=0}^{N-p}\frac{(-1)^{n}s(n+p-1,p-1)}{(n+p-1){N^{n+p-1}}}\left(\sum_{j=0}^{n+p-1}\frac{N^j}{j!}\right) \\
 &=& e^{-N}\sum_{n=0}^{N-p}\frac{(-1)^{n}s(n+p-1,p-1)}{(n+p-1)(n+p-1)!}\left(\sum_{j=0}^{n+p-1}\frac{(n+p-1)!}{j!{N^{n+p-1-j}}}\right) \\
&=&e^{-N}\sum_{n=0}^{N-p}\frac{(-1)^{n}s(n+p-1,p-1)}{(n+p-1)(n+p-1)!}\left(\sum_{j=0}^{n+p-1}\frac{(j+1)(j+2)\cdots(n+p-1)}{{N^{n+p-1-j}}}\right)
\\
 &\leq& e^{-N}\sum_{n=0}^{N-p}\frac{(-1)^{n}s(n+p-1,p-1)}{(n+p-1)(n+p-1)!}\left(\sum_{j=0}^{n+p-1}\left(\frac{n+p-1}{N}\right)^j\right)
\\
 &=&
 e^{-N}\sum_{n=0}^{N-p}\frac{(-1)^{n}s(n+p-1,p-1)}{(n+p-1)(n+p-1)!}\left(\frac{1-\left(\frac{n+p-1}{N}\right)^{n+p}}{1-\frac{n+p-1}{N}}\right)\\
&\leq&
Ne^{-N}\sum_{n=0}^{N-p}\frac{(-1)^{n}s(n+p-1,p-1)}{(n+p-1)(n+p-1)!}.
\end{eqnarray*}
Once again using  (\ref{upper-stir1}) we obtain the following upper
bound of $\epsilon_2$
\begin{eqnarray*}
\epsilon_{2} & \leq& \displaystyle{
Ne^{-N}\sum_{n=0}^{N-p}\frac{(\log(n+p-2))^{p-2}}{(n+p-1)^2(p-2)!}\left(1+\frac{p-2}{\log
(n+p-2)}\right)}\\
&\leq&\displaystyle{\frac{\nu Ne^{-N}}{(p-2)!}},\\
\end{eqnarray*}
where
$\displaystyle{\nu=\sum_{n=0}^{\infty}\frac{(\log(n+p-2))^{p-2}}{(n+p-1)^2}\left(1+\frac{p-2}{\log
(n+p-2)}\right)}.$
\end{proof}

\section{The \texorpdfstring{$S_n^{p}$}- sequence and  Hurwitz zeta function}\label{Hurwitz}
The  Hurwitz zeta function (see \cite{Titchmarsh}) is defined by
\begin{equation*}
\zeta(s,x)= \sum_{n=0}^{\infty}\frac{1}{(n+x)^s}\;\;(0<x\le1,\;\;\Re(s)>1).
\end{equation*}
In particular, $\zeta(s,x)$ is reduced to  the Riemann zeta function $\zeta(s)$ when $x=1$,
\begin{equation*}
\zeta(s)= \sum_{n=1}^{\infty}\frac{1}{n^s}\;\;(\Re(s)>1).
\end{equation*}
Moreover for $x=\frac{1}{2},$ we have   
\begin{equation}\label{Sondow}
    \zeta\left(s,\frac{1}{2}\right)=(2^s-1)\zeta(s).
\end{equation}
The  Hurwitz zeta function, has  an integral representation
\begin{equation*}
\zeta(s,x)=\frac{1}{\Gamma(s)}\int_{0}^{+\infty}\frac{e^{-xt}t^{s-1}}{1-e^{-t}}dt.
\end{equation*}
\begin{theorem}
Let $p=2,3,...$, then for $0<x\le1$, we have
\begin{equation}\label{zethurw}
\zeta(p+1,x) =\frac{1}{p!}\sum_{n=0}^{\infty}(-1)^n\sum_{k=0}^{n}\frac{(n-k)!(k+p-1)!S_{n-k}^{k+p}}{(n-k+1)k!}(x-1)^{k}.
\end{equation}
\end{theorem}
\begin{proof} 

Taking the $(p-1)$th derivative of both sides of (\ref{g-int-1}), we get
\begin{equation*}
g_n^{(p-1)}(x)= \frac{(-1)^{n+p-1}}{n!}\int_{0}^{+\infty}(1-e^{-t})^n e^{-xt}t^{p-1}dt,
\end{equation*}
from which we obtain that
\begin{equation*}
 \frac{1}{p!}\sum_{n=0}^{\infty}\frac{(-1)^{n+p-1}n!}{n+1}g_n^{(p-1)}(x)=\frac{1}{p!} \int_{0}^{+\infty}\sum_{n=0}^{\infty}\frac{(1-e^{-t})^n}{n+1} e^{-xt}t^{p-1}dt.
\end{equation*}
Using the fact that
\begin{equation*}
  -\frac{\ln(1-u)}{u}=\sum_{n=0}^{\infty}\frac{u^n}{n+1}, \quad |u|<1,
\end{equation*}
we get
\begin{equation*}
 \frac{1}{p!}\int_{0}^{+\infty}\sum_{n=0}^{\infty}\frac{(1-e^{-t})^n}{n+1} e^{-xt}t^{p-1}dt= \frac{1}{p!}\int_{0}^{+\infty}\frac{e^{-xt}t^{p}}{1-e^{-t}}dt=\zeta(p+1,x) .
\end{equation*}
Hence
\begin{equation}\label{hurw-gn}
\zeta(p+1,x) =\frac{1}{p!}\sum_{n=0}^{\infty}\frac{(-1)^{n+p-1}n!}{n+1}g_n^{(p-1)}(x).
\end{equation}
Differentiating $p-1$ times both sides of (\ref{gn-stir})  we obtain
\begin{equation*}
g_n^{(p-1)}(x)=  \sum_{k=p-1}^{+\infty}\frac{(-1)^kk!S_n^{k+1}}{(k-p+1)!}(x-1)^{k-p+1}.
\end{equation*}
Change $k$ to $k+p-1$  in the  previous sum, then we get
\begin{equation*}
g_n^{(p-1)}(x)=  \sum_{k=0}^{+\infty}\frac{(-1)^{k+p-1}(k+p-1)!S_n^{k+p}}{k!}(x-1)^{k}.
\end{equation*}
Substitute this in the  expression (\ref{hurw-gn}) to get
\begin{equation*}
\zeta(p+1,x) =\frac{1}{p!}\sum_{n=0}^{\infty}\sum_{k=0}^{+\infty}\frac{(-1)^{n+k}n!(k+p-1)!S_n^{k+p}}{(n+1)k!}(x-1)^{k},
\end{equation*}
which can be written as
\begin{equation*}
\zeta(p+1,x) =\frac{1}{p!}\sum_{n=0}^{\infty}(-1)^n\sum_{k=0}^{n}\frac{(n-k)!(k+p-1)!S_{n-k}^{k+p}}{(n-k+1)k!}(x-1)^{k}.
\end{equation*}
\end{proof}
Consequently, we refined the following formula,
due to Hasse (see \cite{Hasse}).
\begin{corollary}
For $p\geq 2,$ we have
\begin{equation}\label{zet-hass}
  p\zeta(p+1)=\sum_{n=0}^{\infty}\frac{(-1)^nn!}{n+1}S_n^p.
\end{equation}
\end{corollary}
\begin{proof} The   expansion (\ref{zet-hass}) can  be deduced
directly from  (\ref{zethurw})  by setting $x=1$. 
\end{proof}

\section{The \texorpdfstring{$S_n^{p}$}- sequence and polylogarithm function  }\label{Poly-Log}
We recall  that  the polylogarithm is defined by the series, see \cite{andrews}
$$
\mathrm{Li}_{p} (x):=\sum_{n=1}^{\infty} \frac{x^{n}}{n^{p}} \text {
for }|x| \leq 1, p=2,3, \ldots
$$
The polylogarithm function can be expanded in terms of
generalized Stirling numbers of the second kind.
\begin{theorem}
The generalized Stirling numbers of the second kind  , have as
'horizontal' generating function:
\begin{equation}\label{stir-poly}
 \mathrm{Li}_{p}( t )=\sum_{n=0}^{\infty}(-1)^nn!S_n^p\left(\frac{t}{t-1}\right)^{n+1},\;\;  t<\frac{1}{2}.
\end{equation}
\end{theorem}
\begin{proof}
Using the following decay rate of the generalized Stirling numbers:
\begin{equation}\label{snp-estm1-2}
\forall\;\;n\in\mathbb{N},\;\;|S_n^p|\le \frac{1
}{n!}.
\end{equation}
 it follows that the series $ \displaystyle{\sum_{n=0}^{\infty}n!S_n^px^n}$
 is absolutely convergent when $0\le x<1$.
Using (\ref{stir-ln2}), and after interchange the order of summation and integration, we obtain
\begin{eqnarray*}
 \sum_{n=0}^{\infty}(-1)^nn!S_n^px^n &=& \frac{(-1)^{p-1}}{(p-1)!}\int_{0}^{1}\left[\sum_{n=0}^{\infty}\left(-x(u-1)\right)^n \right]\ln^{p-1}(u)du \\
  &=& \frac{(-1)^{p-1}}{(p-1)!}\int_{0}^{1} \frac{\ln^{p-1}(u)}{1+x(u-1)}du \\
   &=& \frac{(-1)^{p-1}}{(p-1)!(1+x)}\int_{0}^{1} \frac{\ln^{p-1}(u)}{1+\frac{xu}{1-x}}du \\
   &=& \frac{(-1)^{p-1}}{(p-1)!(1-x)}\sum_{k=0}^{\infty}\left(\frac{x}{x-1}\right)^k\int_{0}^{1} u^k\ln^{p-1}(u)du. \\
\end{eqnarray*}
Now using the following well known equality
\begin{equation}\label{euler-ln}
 \int_{0}^{1}u^k\ln(u)^{p-1}du= \frac{(-1)^{p-1}(p-1)!}{(k+1)^p},\;\;p=1,2,3,....
\end{equation}
we obtain
\begin{eqnarray*}
  \sum_{n=0}^{\infty}(-1)^nn!S_n^px^n &=& \frac{1}{(1-x)}\sum_{k=0}^{\infty}\frac{\left(\frac{x}{x-1}\right)^k}{(k+1)^p} \\
   &=& \frac{1}{x}\sum_{k=0}^{\infty}\frac{\left(\frac{x}{x-1}\right)^{k+1}}{(k+1)^p}\\
    &=&\frac{1}{x}\mathrm{Li}_{p}\left(\frac{x}{x-1}\right).
\end{eqnarray*}
If we replace $\frac{x}{x-1}$ by $t$ in the last expression, we
obtain the result (\ref{stir-poly}).
\end{proof}
Consequently, we refined the following result due to Sondow (see \cite{Sondow}).
\begin{corollary} For $p\geq 2,$ we have
\begin{equation}\label{zet-Sondow}
 (1-2^{1-p})\zeta(p)=\sum_{n=0}^{\infty}\frac{(-1)^nn!}{2^{n+1}}S_n^p.
\end{equation}
\end{corollary}
\begin{proof}
By choosing $t = -1$ in (\ref{stir-poly}), we have
\begin{equation*}
   \mathrm{Li}_{p}( -1 )=\sum_{n=0}^{\infty}\frac{n!(-1)^{n+1}}{2^{n+1}}S_n^p.
\end{equation*}
Since,
\begin{equation*}
  \mathrm{Li}_{p}( -1 )= \sum_{n=1}^{\infty} \frac{(-1)^{n}}{n^{p}}=\sum_{n=1}^{\infty} \frac{1}{(2n)^{p}}-\sum_{n=0}^{\infty} \frac{1}{(2n+1)^{p}}=
  \frac{1}{2^p}\left(\zeta(p)-\zeta(p,\frac{1}{2}) \right)
\end{equation*}
and using (\ref{Sondow}), 
 we deduce the formula (\ref{zet-Sondow}). 
\end{proof}

\section{The  \texorpdfstring{$S_n^{p}$ }-   sequence and  harmonic numbers}\label{HarN}
The $p-$order harmonic numbers are defined by
\begin{equation*}
H_{n}^{(p)}=
\sum_{k=0}^{n-1}\frac{1}{(k+1)^p},\;\;n\in\mathbb{N}^*,\;\;p\in\mathbb{N},
\end{equation*}
with the convention $H_{0}^{(p)}=1,\;\;p\in\mathbb{N}$
\begin{proposition}
The generalized Stirling numbers  of the second kind are related to the the  harmonic numbers
in the following manner:
\begin{equation}\label{snp-harmo}
 S_n^{p}=\frac{(-1)^{n}}{n!}\left(\sum_{k=0}^{n}(-1)^{k+1} \binom{n+1}{k} H_{k}^{(p)}+(-1)^{n}H_{n+1}^{(p)}\right).
\end{equation}
\end{proposition}
\begin{proof}
First note that, $ \displaystyle\frac{1}{(k+1)^p}=H_{k+1}^{(p)}-H_{k}^{(p)}$, so
\begin{eqnarray*}
\sum_{k=0}^{n}\frac{(-1)^{k}\binom{n}{k}}{(k+1)^p}
   &=& \sum_{k=0}^{n}(-1)^{k}\binom{n}{k}H_{k+1}^{(p)}-\sum_{k=0}^{n}(-1)^{k}\binom{n}{k}H_{k}^{(p)} \\
   &=& \sum_{k=0}^{n}(-1)^{k}\binom{n}{k}H_{k+1}^{(p)}-\sum_{k=0}^{n-1}(-1)^{k+1}\binom{n}{k+1}H_{k+1}^{(p)}-1 \\
   &=& \sum_{k=0}^{n-1}(-1)^{k}\left(\binom{n}{k}+\binom{n}{k+1}\right)H_{k+1}^{(p)}+(-1)^nH_{n+1}^{(p)}-1,
\end{eqnarray*}
 by   Pascal's identity
\begin{equation*}
 \binom{n}{k}+\binom{n}{k+1}=\binom{n+1}{k+1},
\end{equation*}
and the change  of the indices  in the summation, we deduce (\ref{snp-harmo}).
\end{proof}
The (exponential) complete Bell polynomials $Y_n=Y_n(x_1, x_2,...,x_n)$ are
defined by (see \cite{Comtet}):
\begin{equation}\label{bell-genera}
  \exp\left(\sum_{m=1}^{\infty}x_m\frac{t^m}{m!} \right)=1+\sum_{n=1}^{\infty}Y_n(x_1, x_2,...)\frac{t^n}{n!}.
\end{equation}
\begin{theorem}
 The $S_n^{p}$-sequence can also be expressed in terms of harmonic numbers  as
\begin{equation}\label{snp-bell}
 S_n^{p+1}= \frac{(-1)^n}{(n+1)!p!} y_{p}(n),
\end{equation}
where
\begin{equation}\label{bell-yn}
y_{p}(n)=
Y_{p}\left(H_{n+1}^{(1)},1!H_{n+1}^{(2)},2!H_{n+1}^{(3)},...,(p-1)!H_{n+1}^{(p)}\right).
\end{equation}
\end{theorem}
\begin{proof}
First, we may write
\begin{eqnarray*}
  \frac{(-1)^{n}}{(1-t)_{n+1}} &=& \frac{(-1)^{n}}{(n+1)!(1-t)(1-\frac{t}{2})(1-\frac{t}{3})...(1-\frac{t}{n+1})} \\
   &=& \frac{(-1)^{n}}{(n+1)!}\exp\left(-\sum_{k=0}^{n}\ln(1-\frac{t}{k+1})  \right),
\end{eqnarray*}
expanding $\displaystyle\ln(1-\frac{t}{k+1})$  in  power series of $t$ in the second term,  and interchanging the order of summation, we obtain
\begin{eqnarray*}
  \frac{(-1)^{n}}{(1-t)_{n+1}}
   &=& \frac{(-1)^{n}}{(n+1)!}\exp\left(\sum_{p=1}^{\infty}\frac{H_{n+1}^{(p)}}{p}t^{p}  \right)\\
   &=& \frac{(-1)^{n}}{(n+1)!}\exp\left(\sum_{p=1}^{\infty}\frac{(p-1)!H_{n+1}^{(p)}}{p!}t^{p}  \right).\\
\end{eqnarray*}
From (\ref{bell-genera}) and (\ref{bell-yn}), we have
\begin{equation}\label{expression}
 \frac{(-1)^{n}}{(1-t)_{n+1}}=\frac{(-1)^{n}}{(n+1)!}\left( 1+\sum_{p=1}^{\infty}y_p(n)\frac{t^p}{p!}\right).
\end{equation}

Identifying (\ref{expression}) with (\ref{generatp-1}), we deduce (\ref{snp-bell}).
\end{proof}
We list below the first five  complete Bell polynomials  (see \cite{Comtet}  p.307)
\begin{eqnarray*}
 Y_{1}(x_1)  &=& x_1 \\
 Y_{2}(x_1,x_2) &=&x_1^2 +x_2  \\
 Y_{3}(x_1,x_2,x_3) &=& x_1^3+3x_1x_2+x_3 \\
 Y_{4}(x_1,x_2,x_3,x_4) &=&x_1^4+6x_1^2x_2+4x_1x_3+3x_2^2+x_4 \\
 Y_{5}(x_1,x_2,x_3,x_4,x_5) &=& x_1^5+10x_1^3x_2+10x_1^2x_3+15x_1x_2^2+5x_1x_4+10x_2x_3+x_5.
\end{eqnarray*}
Thus, we deduce that
\begin{eqnarray}
  S_n^2 &=& \frac{(-1)^n}{(n+1)!}H_{n+1}^{(1)} \label{Sn2} \\
  S_n^3  &=&\frac{(-1)^n}{2(n+1)!}\left[\left(H_{n+1}^{(1)}\right)^2 +H_{n+1}^{(2)}\right]\nonumber  \\
 S_n^4  &=&\frac{(-1)^n}{6(n+1)!}\left[\left(H_{n+1}^{(1)}\right)^3 +3H_{n+1}^{(1)}H_{n+1}^{(2)}+2H_{n+1}^{(3)}\right] \nonumber\\
 S_n^5  &=&\frac{(-1)^n}{24(n+1)!}\left[\left(H_{n+1}^{(1)}\right)^4 +6\left(H_{n+1}^{(1)}\right)^2H_{n+1}^{(2)}+8H_{n+1}^{(1)}H_{n+1}^{(3)}+3\left(H_{n+1}^{(2)}\right)^2+6H_{n+1}^{(4)}\right] \nonumber\\
S_n^6  &=&
\frac{(-1)^n}{120(n+1)!}\left[\left(H_{n+1}^{(1)}\right)^5
+10\left(H_{n+1}^{(1)}\right)^3H_{n+1}^{(2)}+20\left(H_{n+1}^{(1)}\right)^2H_{n+1}^{(3)}\right.
\nonumber\\
&&\left.+15H_{n+1}^{(1)}\left(H_{n+1}^{(2)}\right)^2+30H_{n+1}^{(1)}H_{n+1}^{(4)}+20H_{n+1}^{(2)}H_{n+1}^{(3)}+24H_{n+1}^{(5)}\right]\nonumber.
\end{eqnarray}
Note here that the author of \cite{Shen} proved similar equalities
of harmonic sums  without noting that such harmonic sums are related
to the generalized Stirling numbers of the second kind.
\section{The \texorpdfstring{$S_n^{p}$}- sequence and  multiple sums}\label{Multi-Sum}
In this section we show some relations between  the generalized Stirling numbers and multiple sums.
\begin{proposition}
The $S_n^{p}$- sequence can be rewritten  through a multiple sum
\begin{equation}\label{S_n^{p}-msum}
S_n^{p}=\frac{(-1)^{n+p}}{(n+1)!(p-1)!}
\sum_{k_1=0}^{n}\frac{1}{k_1+1}\sum_{k_2=0}^{k_1}\frac{1}{k_2+1}...\sum_{k_{p-1}=0}^{k_{p-2}}\frac{1}{k_{p-1}+1},\;\;p\geq
2.
\end{equation}
\end{proposition}
\begin{proof} The proof is by induction on $p$. For $p=2$,
(\ref{S_n^{p}-msum}) is  true, in fact we have (\ref{Sn2}).

Thus, the formula is clearly true for $p = 2$. Assume
it is true for $2,3,...,p $.  By (\ref{stir-ln2}), we have
\begin{eqnarray*}
  S_n^{p+1} &=& \frac{(-1)^{n+p}}{n!p!}\int_{0}^{1}u^n\ln^{p-1}(1-u)\ln(1-u)du\\
   &=& \frac{(-1)^{n+p+1}}{n!p!}\sum_{k=0}^{\infty}\frac{1}{k+1}\int_{0}^{1}u^{n+k+1}\ln^{p-1}(1-u)du \\
  &=&  \frac{(-1)^{n+p+1}}{n!p!}\sum_{k=0}^{\infty}\frac{1}{k+1}(-1)^{n+k+p}(n+k+1)!(p-1)!S_{n+k+1}^{p} \\
   &=&  \frac{(-1)^{n+p+1}}{n!p!}\sum_{k=0}^{\infty}\frac{1}{(k+1)(n+k+2)}\sum_{k_1=0}^{n+k+1}\frac{1}{k_1+1}\sum_{k_2=0}^{k_1}\frac{1}{k_2+1}...\sum_{k_{p-1}=0}^{k_{p-2}}\frac{1}{k_{p-1}+1} . \\
\end{eqnarray*}
 Using the fact that:
\begin{equation}\label{decomp}
 \frac{1}{(k+1)(n+k+2)}=\frac{1}{n+1}\left(\frac{1}{k+1}-\frac{1}{n+k+2} \right),
\end{equation} we obtain
 \begin{eqnarray*}
 S_n^{p+1}   &=&\frac{(-1)^{n+p+1}}{(n+1)!p!}\sum_{k=0}^{\infty}\frac{1}{k+1}\sum_{k_1=0}^{n+k+1}\frac{1}{k_1+1}\sum_{k_2=0}^{k_1}\frac{1}{k_2+1}...\sum_{k_{p-1}=0}^{k_{p-2}}\frac{1}{k_{p-1}+1} \\
    &&  - \frac{(-1)^{n+p+1}}{(n+1)!p!}\sum_{k=0}^{\infty}\frac{1}{n+k+2}\sum_{k_1=0}^{n+k+1}\frac{1}{k_1+1}\sum_{k_2=0}^{k_1}\frac{1}{k_2+1}...\sum_{k_{p-1}=0}^{k_{p-2}}\frac{1}{k_{p-1}+1} \\
    &=&  \frac{(-1)^{n+p+1}}{(n+1)!p!}\sum_{k=0}^{n}\frac{1}{k+1}\sum_{k_1=0}^{k}\frac{1}{k_1+1}\sum_{k_2=0}^{k_1}\frac{1}{k_2+1}...\sum_{k_{p-1}=0}^{k_{p-2}}\frac{1}{k_{p-1}+1}.
 \end{eqnarray*}
\end{proof}
\begin{proposition}
The $S_n^{p}$-sequence can be rewritten also through a multiple sum
\begin{equation}\label{desired-equ}
S_n^{p+1}=\frac{(-1)^{n}}{(n+1)!}\sum_{k_1=0}^{p}\frac{1}{2^{p-k_1}}
\sum_{k_2=0}^{k_1}\frac{1}{3^{k_1-k_2}} ...\sum_{k_{n-1}=0}^{k_{n-2}}\frac{1}{n^{k_{n-2}-k_{n-1}}}\sum_{k_{n}=0}^{k_{n-1}}\frac{1}{(n+1)^{k_{n}}}.
\end{equation}
\end{proposition}
\begin{proof}
From (\ref{generatp}), we have
\begin{equation*}
  \frac{(-1)^{n}}{(1-t)_{n+1}}=\frac{(-1)^{n}}{(n+1)!(1-t)(1-\frac{t}{2})(1-\frac{t}{3})...(1-\frac{t}{n+1})}.
\end{equation*}
Now, for each $k=1,2,...,n+1$, we expand  $(1-\frac{t}{k})^{-1}$ in a power series of $t$, we get
\begin{equation*}
 \frac{(-1)^{n}}{(1-t)_{n+1}}=\frac{(-1)^{n}}{(n+1)!}\sum_{k_1=0}^{\infty}\sum_{k_2=0}^{\infty}
 ...\sum_{k_n=0}^{\infty}\frac{t^{k_1+k_2+...+k_{n+1}}}{2^{k_1}3^{k_2}...(n+1)^{k_{n+1}}}.
 \end{equation*}
By multiple Cauchy product of infinite series formula, we have
\begin{equation*}
 \frac{(-1)^{n}}{(1-t)_{n+1}}=\frac{(-1)^{n}}{(n+1)!}\sum_{k_1=0}^{\infty}t^{k_1}\sum_{k_2=0}^{k_1}\frac{1}{2^{k_1-k_2}}
\sum_{k_3=0}^{k_2}\frac{1}{3^{k_2-k_3}} ...\sum_{k_{n}=0}^{k_{n-1}}\frac{1}{n^{k_{n-1}-k_{n}}}\sum_{k_{n+1}=0}^{k_{n}}\frac{1}{(n+1)^{k_{n+1}}}.
 \end{equation*}
 by identification of the last  equation and (\ref{generatp}), we obtain (\ref{desired-equ}).
\end{proof}
\section{The  generalized $q$-Stirling numbers of the second kind }
Throughout this section, we assume that $0<q<1$ and  we note
$$[x]_q=\frac{1-q^x}{1-q},\quad x\in \mathbb{C}.$$ For a complex number a, the $q$-shifted factorials are defined by
\begin{equation*}
(a;q)_n :=
\begin{cases}
1, & \text{if } n = 0, \\
\prod_{k=0}^{n-1} (1 - aq^k), & \text{if } n \in \mathbb{N}.
\end{cases}
\end{equation*}

\begin{equation*}
 (a;q)_\infty :=  \displaystyle{\prod_{k=0}^{\infty}} (1 - aq^k) 
\end{equation*}
The $q$-factorials are defined by
\begin{equation*}
 [n]_q! :=
\begin{cases}
1, & \text{if } n = 0, \\
[n]_q [n - 1]_q \cdots [2]_q [1]_q, & \text{if } n \in \mathbb{N},
\end{cases}   
\end{equation*}
and  hence the $q$-binomial coefficient is given by
\begin{equation*}
  \begin{bmatrix} n \\ k \end{bmatrix}_q :=
\frac{[n]_q!}{[n - k]_q! \, [k]_q!}.  
\end{equation*}
Using the $q$-difference operators $\Delta^{k}$ defined by  the recurrence relation (see \cite{Srivastava2011ZetaAQ})
\begin{equation}\label{qdiff}
 \Delta^0 f(x)=f(x),\;\;\Delta^{k+1} f(x)=\Delta^{k} f(x+1)-q^k\Delta^{k} f(x),\;\;k=1,2,..., 
\end{equation}
one may define  the  generalized $q$-Stirling numbers of the second kind as follow
\begin{equation*}
S_q(n,p)= \frac{(-1)^{n+p}}{[n]_q!}\Delta^{n}\left(\left([x-n-1]_q\right)^{-p} \right)_{x=0},\;\;n,p=0,1,2,....
\end{equation*}
The representation of the $q$-difference operators $\Delta^{n}$ 
\begin{equation*}
\Delta^{n}f(x)= \displaystyle\sum_{k=0}^{n}(-1)^{k}q^{\binom{k}{2}}\left[\begin{array}{l}
n  \\
k
\end{array}\right]_{q}f(x+n-k),\;\;n=0,1,2,....
\end{equation*}
 leads to an explicit formula for the  generalized $q$-Stirling numbers of the second kind:
\begin{equation}\label{stir-expl}
S_q(n,p)=\frac{(-1)^n}{[n]_q!}\displaystyle\sum_{k=0}^{n}\frac{(-1)^{k}q^{\binom{k}{2}+p(k+1)}\left[\begin{array}{l}
n  \\
k
\end{array}\right]_{q}}{\left([k+1]_q\right)^p},\;\;n,p=0,1,2,....
\end{equation}

A $q$-analogue of the  exponential function  is defined by
\begin{equation*}
e_q(z):=\sum_{n=0}^{+\infty}\frac{z^n}{[n]_q!}=\frac{1}{((1-q)z;q)_\infty},\;\;|z|<1. 
\end{equation*}
The $q$-hypergeometric series ${}_{p} \Phi_{p}$ is defined by
\begin{equation*}
{}_{r}\Phi_{s}\left(a_1, \ldots, a_r ; b_1, \ldots, b_s ; q, z\right)=\sum_{k=0}^{+\infty}  \frac{(a_1;q)_k...(a_r;q)_k}{(b_1;q)_k...(b_s;q)_k}\left((-1)^kq^{\binom{k}{2}}\right)^{1+s-r}\frac{z^k}{(q;q)_k}.
\end{equation*}
\begin{theorem}\label{q-generat1} The  generating function of the generalized $q$-Stirling numbers of  second kind $S_q(n,p)$  is given by:
\begin{equation}\label{part1}
e_q({-q^{-p}t}){}{ }_{p} \Phi_{p}\left(q, \ldots, q ; q^2, \ldots, q^2 ; q, -(1-q)t\right)=\sum_{n=0}^{+\infty}
q^{-p(n+1)}S_q(n,p)t^{n},\;t\in\mathbb{R}.
\end{equation}
\end{theorem} 
\begin{proof} 

Expanding both functions of the left hand side of
(\ref{part1}) into their power series, 
we obtain:
\begin{eqnarray*}
e_q({-q^{-p}t}) {}_{p} \Phi_{p}\left(q, \ldots, q ; q^2, \ldots, q^2 ; q, -(1-q)t\right)
  &=&  \sum_{n=0}^{+\infty}\frac{(-1)^nq^{-pn}t^n}{[n]_q!}\\  &&\times\sum_{k=0}^{\infty}q^{\binom{k}{2}}  
  \frac{\left(q ; q\right)_{k}... \left(q ; q\right)_{k}}{\left(q^2; q\right)_{k} ...\left(q^2 ; q\right)_{k}} \frac{(1-q)^kt^{k}}{(q ; q)_{k}}  \\
  &=&  \sum_{n=0}^{+\infty}\frac{(-1)^nq^{-pn}}{[n]_q!}\sum_{k=0}^{+\infty}\frac{q^{\binom{k}{2}} t^{n+k}}{[k]_q!\left([k+1]_q\right)^p} \\
&=&  \sum_{n=0}^{+\infty}q^{-pn}t^n \sum_{k=0}^{n}\frac{(-1)^{n-k}q^{\binom{k}{2}+pk} }{[n-k]_q![k]_q!\left([k+1]_q\right)^p} \\
&=&  \sum_{n=0}^{+\infty}q^{-p(n+1)}t^n \sum_{k=0}^{n}\frac{(-1)^{n-k}q^{\binom{k}{2}+p(k+1)} }{[n-k]_q![k]_q!\left([k+1]_q\right)^p} \\
 &=& \sum_{n=0}^{+\infty}q^{-p(n+1)}S_q(n,p)t^n.
\end{eqnarray*}
\end{proof}

By using the fact that
\begin{align*}
& \left(q^{-n} ; q\right)_{k}=\frac{(q ; q)_{n}}{(q ; q)_{n-k}}(-1)^{k} q^{\binom{k}{2}-n k} \quad(n, k \in \mathbb{Z}),  
\end{align*}
we have the following representation of the generalized Stirling numbers of the second kind in terms of
$q$-hypergeometric series:

\begin{lemma}
The $S_q(n,p)$ numbers
can be expressed as:
\begin{equation*}
S_q(n,p)=\frac{(-1)^n }{[n]_q!} {}{ }_{p+1} \Phi_{p}\left(q^{-n},q, \ldots, q ; q^2, \ldots, q^2 ; q, q^{n+p+1}\right).   
\end{equation*}
\end{lemma}
Our aim now is to give a relation between q-generalized Stirling numbers and the function given by 
\begin{equation}\label{gn-q}
  g_{n,q}(x)=\dfrac{(-1)^n}{(qx;q)_{n+1}}.  
\end{equation}
Thus, we need the following Lemma.
\begin{lemma}
The function $g_{n,q}$ 
can be expressed as
\begin{equation}\label{g_{n,q}0}
  g_{n,q}(x)=\frac{(-1)^{n+1}q^{-1}}{(q;q)_{n}}\displaystyle\sum_{k=0}^{n}\frac{(-1)^{k}q^{\binom{k}{2}}\left[\begin{array}{l}
n  \\
k
\end{array}\right]_{q}}{\left(x-q^{-k-1}\right)}.
\end{equation}
\end{lemma}
\begin{proof}
We have:
\begin{align*}
  g_{n,q}(x)&=\frac{(-1)^n}{(1-qx)(1-q^{2}x)...(1-q^{n+1}x)}\\
  &=(-1)^n\sum_{k=0}^{n}\frac{a_k}{x-q^{-k-1}}
\end{align*}
where 
\begin{align*}
a_k&= \displaystyle{\lim_{x\rightarrow q^{-k-1}}\frac{x-q^{-k-1}}{(1-qx)(1-q^{2}x)...(1-q^{n+1}x)}}\\
&=\frac{-q^{-k-1}}{(1-q^{-k})(1-q^{-k+1})...(1-q^{-1})(1-q)...(1-q^{n-k})}\\
&=\frac{(-1)^{k+1}q^{-k-1}q^{\binom{k+1}{2}}}{(q;q)_{k}(q;q)_{n-k}}\\
\end{align*}
\end{proof}



\begin{lemma}
We have:
\begin{equation}\label{derivp}
  S_q(n,p)=\frac{q(1-q)^{n+p}}{(p-1)!}g_{n,q}^{(p-1)}(1)
\end{equation}
\end{lemma}
\begin{proof}
Differentiating both sides of (\ref{g_{n,q}0}) $(p-1)$ times and taking $x=1,$ we obtain

\begin{align*}
  g_{n,q}^{(p-1)}(1) &=\frac{(-1)^{n+p}q^{-1}(p-1)!}{(q;q)_{n}}\displaystyle\sum_{k=0}^{n}\frac{(-1)^{k}q^{\binom{k}{2}}\left[\begin{array}{l}
n  \\
k
\end{array}\right]_{q}}{\left(1-q^{-k-1}\right)^p}\\
&=\frac{(-1)^{n}q^{-1}(p-1)!}{(q;q)_{n}}\displaystyle\sum_{k=0}^{n}\frac{(-1)^{k}q^{\binom{k}{2}}q^{p(k+1)}\left[\begin{array}{l}
n  \\
k
\end{array}\right]_{q}}{\left(1-q^{k+1}\right)^p}\\
&=\frac{q^{-1}(p-1)!}{(1-q)^{n+p}}S_q(n,p).
\end{align*}
\end{proof}
\begin{theorem}
The generalized $q$-Stirling numbers  of the second kind
can be defined by the following 'vertical' generating function:
\begin{equation}\label{generatp-q}
  \frac{(-1)^n}{(q(1+t);q)_{n+1}}=\frac{q^{-1}}{(1-q)^{n+1}}\sum_{p=0}^{+\infty}\frac{S_q(n,p+1)}{(1-q)^p}t^p.
\end{equation}
\end{theorem}
\begin{proof} To prove (\ref{generatp-q}), we write $g_{n,q}(x) $ as a power
series in $(x-1)$. We use Taylor's theorem to write
\begin{equation*}
g_{n,q}(x)=  \sum_{p=0}^{+\infty}\frac{g_{n,q}^{(p)}(1)}{p!}(x-1)^p.
\end{equation*}
Hence, from equation (\ref{derivp}), we have
\begin{equation}\label{gn-stirq}
g_{n,q}(x)=  \frac{q^{-1}}{(1-q)^{n+1}}\sum_{p=0}^{+\infty}S_q(n,p+1)\left(\frac{x-1}{1-q}\right)^p.
\end{equation}
Putting $t=x-1$ in (\ref{gn-stirq}), we get (\ref{generatp-q}).
\end{proof}
\begin{remark}
Take $x=q$ in (\ref{gn-stirq}), we get
\begin{equation}\label{q2stir}
   \frac{1}{[n+2]_q!}=q^{-1}\sum_{p=0}^{+\infty}(-1)^{n+p}S_q(n,p+1). 
  \end{equation}
\end{remark}
\begin{theorem}
The generalized $q$-Stirling numbers of the second kind satisfy the recurrence relation
\begin{equation}\label{recq}
 [n+1]_q S_q(n,p+1)=q^{n+1}S_q(n,p)-S_q(n-1,p+1);\;n=1,2,3,...;\; p=0,1,2,...,
\end{equation}  
with   
 $$ S_q(n,0)=0,\;\; \;\ \; S_q(0,p)= q^p,
$$
\end{theorem}

\begin{proof}


Using the fact that 
\begin{equation*}
\frac{(-1)^{n}}{(q(1+t);q)_{n+1}} =-(1-(1+t)q^{n+2})\frac{(-1)^{n+1}}{(q(1+t);q)_{n+2}}
\end{equation*}
and by (\ref{generatp-q}), we have
\begin{equation*}
(1-q)\sum_{p=0}^{+\infty}\frac{S_q(n,p+1)}{(1-q)^p}t^p=  -(1-(1+t)q^{n+2})\sum_{p=0}^{+\infty}\frac{S_q(n+1,p+1)}{(1-q)^p}t^p,
\end{equation*}
from which, we deduce
\begin{eqnarray*}
\sum_{p=0}^{+\infty}\frac{(1-q)S_q(n,p+1)+(1-q^{n+2})S_q(n+1,p+1)}{(1-q)^p}t^p  &=&tq^{n+2}\sum_{p=0}^{+\infty}\frac{S_q(n+1,p+1)}{(1-q)^p}t^p \\
   &=& q^{n+2}(1-q)\sum_{p=0}^{+\infty}\frac{S_q(n+1,p+1)}{(1-q)^{p+1}}t^{p+1} \\
   &=& q^{n+2}(1-q)\sum_{p=1}^{+\infty}\frac{S_q(n+1,p)}{(1-q)^{p}}t^{p}.
\end{eqnarray*}
Consequently, we have
\begin{equation*}
 S_q(n,p+1)+[n+2]_qS_q(n+1,p+1)=q^{n+2} S_q(n+1,p).
\end{equation*}
\end{proof}

\subsection{A Generalized $q$-Zeta Function}
We   define  the $q$-Zeta function $\zeta_q(s)$ in a slightly different way to the definition given in \cite{Srivastava2011ZetaAQ}. Let 
\begin{equation}\label{def-qHurwitz}
\zeta_q(s):=  \sum_{k=0}^{+\infty}\frac{q^{k(s-1)}}{\left([k+1]_q\right)^s}.
\end{equation}
In the case $s=p=2,3,...$, the $q$-Zeta values $\zeta_q(p)$ are  given by  $q$-hypergeometric series:
\begin{equation}\label{qzeta-hyp}
   \zeta_q(p) = {}{ }_{p+1} \Phi_{p}\left(q,q, \ldots, q ; q^{2}, \ldots, q^{2} ; q, q^{p-1}\right). 
\end{equation}
The  $q$-integral of $f(t)$ on $[0, \infty)$ is defined by
\begin{align*}
\int_{0}^{\infty} f(t) d_{q} t:
& =(1-q) \sum_{j=-\infty}^{\infty} f\left(q^{j}\right) q^{j} \quad(0<q<1)
\end{align*}
We have (see \cite{mezlini1}, p. 24, Lemma 5.1)
\begin{lemma}
For $\lambda > 0$ and $s \geq 1$, we have
\begin{equation}\label{q-euler-intg}
 \int_{0}^{\infty} e_q(-\frac{\lambda}{1-q} t)t^{s-1} d_{q} t  =\frac{(1-q)q^{-\binom{s}{2}}\left(q ; q\right)_{s-1}}{\lambda^s} .
\end{equation}
\end{lemma}
From which the following $q-$integral can be obtained:
\begin{proposition}
For $\lambda > 0$ such that  $\lambda q > 1$, we have
  \begin{equation}\label{q-laplace}
\int_{0}^{\infty}  e_q(-\frac{\lambda}{1-q} t) {}{ }_{p} \Phi_{p}\left(q, \ldots, q ; q^2, \ldots, q^2 ; q, -t\right) d_{q} t=\frac{1-q}{\lambda} {}{ }_{p+1} \Phi_{p}\left(q, \ldots, q ; q^2, \ldots, q^2 ; q, \frac{1}{q\lambda}\right). 
  \end{equation}  
\end{proposition}
\begin{theorem}
 The $q$-integral representation of the $q$-zeta function is given by:
 \begin{equation}\label{q-int-zeta}
\zeta_q(p)=\frac{1}{q^{p}(1-q)}\int_{0}^{\infty}  e_q(-\frac{q^{-p}}{1-q} t) {}{ }_{p} \Phi_{p}\left(q, \ldots, q ; q^2, \ldots, q^2 ; q, -t\right) d_{q} t.
  \end{equation}
 \end{theorem}
\begin{proof} 
Take $\lambda=q^{-p}$ in (\ref{q-laplace}) and by (\ref{qzeta-hyp}), we deduce that
\begin{equation*}
\int_{0}^{\infty}  e_q\left(-\frac{q^{-p}}{1-q} t\right) {}{ }_{p} \Phi_{p}\left(q, \ldots, q ; q^2, \ldots, q^2 ; q, -t\right) d_{q} t=q^{p}(1-q){}{ }_{p+1} \Phi_{p}\left(q, \ldots, q ; q^2, \ldots, q^2 ; q,q^{p-1}\right) .
  \end{equation*} 
 \end{proof} 

\begin{proposition}
Let $p \geq 2$ be an integer then: 
\begin{equation}\label{herwaux}
 \zeta_q(p)=(1-q)^p \sum_{n=0}^{+\infty}\binom{n+p-1}{p-1}\frac{q^{n}}{1-q^{n+p-1}}.
\end{equation}
\end{proposition}
\begin{proof}

 It follows immediately from the binomial Theorem that:
 \begin{equation*}
 \frac{1}{(1-x)^p}= \sum_{n=0}^{+\infty}\binom{n+p-1}{p-1}x^{n},
\end{equation*}
 Hence,
 \begin{eqnarray*}
   \zeta_q(p) &=& (1-q)^p  \sum_{k=0}^{+\infty}\frac{q^{k(p-1)}}{\left(1-q^{k+1}\right)^p} \\
    &=& (1-q)^p \sum_{k=0}^{+\infty}q^{k(p-1)}\sum_{n=0}^{+\infty}\binom{n+p-1}{p-1}q^{n(k+1)} \\
    &=&(1-q)^p \sum_{n=0}^{+\infty}\binom{n+p-1}{p-1}q^{n}\sum_{k=0}^{+\infty}q^{k(n+p-1)}\\
    &=& (1-q)^p\sum_{n=0}^{+\infty}\binom{n+p-1}{p-1}\frac{q^{n}}{1-q^{n+p-1}}. 
 \end{eqnarray*}
\end{proof}

\begin{theorem} For $p=3,4,...$ we have
\begin{equation}\label{zetaq15}
 \zeta_q(p) = (1-q)^{p-1}\sum_{n=0}^{+\infty}(-1)^{n}\left(\frac{1-q^{p-2}}{1-q}\right)^{n}\sum_{k=0}^{n}\binom{k+p-1}{p-1}\frac{q^{k-1}(q^{p-1};q)_{k}}{(1-q^{p-2})^k}S_q(k,n-k+1). 
\end{equation}
\end{theorem}
 \begin{proof}
 Remark firstly that by (\ref{gn-q}) we have
 $$\frac{1}{1-q^{n+p-1}}=\frac{(q^{p-1};q)_{n}}{(q^{p-1};q)_{n+1}}=(q^{p-1};q)_{n}(-1)^ng_{n,q}(q^{p-2}).$$

Consequently, using (\ref{herwaux}) together with the previous equalities and  (\ref{gn-stirq}), we obtain
\begin{eqnarray*}
  \zeta_q(p) &=&
  (1-q)^p\sum_{n=0}^{+\infty}\binom{n+p-1}{p-1}q^{n}(q^{p-1};q)_{n}(-1)^ng_{n,q}(q^{p-2})\\
    &=&(1-q)^p\sum_{n=0}^{+\infty}\binom{n+p-1}{p-1}q^{n}(q^{p-1};q)_{n}(-1)^{n}\frac{q^{-1}}{(1-q)^{n+1}}\sum_{k=0}^{+\infty}S_q(n,k+1)\left(\frac{q^{p-2}-1}{1-q}\right)^k \\
   &=&(1-q)^p\sum_{n=0}^{+\infty}(-1)^{n}\sum_{k=0}^{n}\binom{n-k+p-1}{p-1}\frac{q^{n-k-1}(q^{p-1};q)_{n-k}}{(1-q)^{n-k+1}}S_q(n-k,k+1)\left(\frac{1-q^{p-2}}{1-q}\right)^{k} \\ 
   &=&(1-q)^p\sum_{n=0}^{+\infty}(-1)^{n}\sum_{k=0}^{n}\binom{k+p-1}{p-1}\frac{q^{k-1}(q^{p-1};q)_{k}}{(1-q)^{k+1}}S_q(k,n-k+1)\left(\frac{1-q^{p-2}}{1-q}\right)^{n-k} \\
    &=&(1-q)^{p-1}\sum_{n=0}^{+\infty}(-1)^{n}\left(\frac{1-q^{p-2}}{1-q}\right)^{n}\sum_{k=0}^{n}\binom{k+p-1}{p-1}\frac{q^{k-1}(q^{p-1};q)_{k}}{(1-q^{p-2})^k}S_q(k,n-k+1), \\     
 \end{eqnarray*}
which complete the proof of (\ref{zetaq15})
\end{proof}

\bibliographystyle{alpha}
\bibliography{sample}

\end{document}